\documentclass{amsart}
\allowdisplaybreaks[4]
\usepackage{color}

\newcommand{\Z}{\mathbf{Z}}
\newcommand{\R}{\mathbf{R}}
\newcommand{\C}{\mathbf{C}}

\newcommand{{\ba}}{\bf a}
\newcommand{\ve}{\varepsilon}
\newcommand{\la}{\lambda}
\newcommand{\La}{\Lambda}
\newcommand{\ga}{\gamma}
\newcommand{\pa}{\partial}
\newcommand{\ra}{\rightarrow}
\newcommand{\Om}{\Omega}
\newcommand{\del}{\delta}

\newcommand{\na}{\nabla}

\newcommand{\al}{\alpha}

\newcommand{\be}{\begin{equation}}
\newcommand{\ee}{\end{equation}}

\newcommand{\om}{\omega}

\newtheorem{lem}{Lemma}{\bf}{\it}
{\it}{\rm}

\newtheorem{theorem}{Theorem}
\newtheorem{proposition}{Proposition}

\numberwithin{theorem}{section}
\numberwithin{lem}{section}
\numberwithin{hypothesis}{section}
\numberwithin{equation}{section}
\numberwithin{proposition}{section}
\numberwithin{corollary}{section}
\title[Strong Convergence II] {Strong Convergence to the homogenized limit of elliptic equations with random coefficients II}

\author{Joseph G. Conlon \  and \  Arash Fahim}

\address{ (Joseph G. Conlon): University of Michigan, Department of Mathematics, Ann Arbor,
  MI 48109-1109}
\email{conlon@umich.edu}

\address{ (Arash Fahim): University of Michigan, Department of Mathematics, Ann Arbor,
  MI 48109-1109}
\email{fahimara@umich.edu}


\begin{document}

\maketitle

\begin{abstract}
Consider a discrete uniformly elliptic divergence form equation on the $d$ dimensional lattice $\Z^d$ with random coefficients. In \cite{cs}  rate of convergence results in homogenization and estimates on the difference between the averaged Green's function and the homogenized Green's function for random environments which satisfy a Poincar\'{e} inequality were obtained. Here these results are extended to certain environments with long range correlations. These environments are simply related via a convolution to environments which do satisfy a  Poincar\'{e} inequality.
\end{abstract}

\section{Introduction.}
In this paper we continue the study of solutions to divergence form elliptic equations with random coefficients begun in \cite{cs}.  In \cite{cs} we were concerned with solutions  $u(x,\eta,\om)$ to the 
equation
\be \label{A1}
\eta u(x,\eta,\om)+\nabla^*{\bf a}(\tau_x\om)\nabla u(x,\eta,\om)=h(x), \quad x\in \Z^d, \ \om\in\Om,
\ee
where  $\Z^d$ is the $d$ dimensional integer lattice and $(\Om,\mathcal{F},P)$ is a probability space equipped with  measure preserving  translation operators  $\tau_x:\Om\ra\Om, \ x\in\Z^d$. The function  ${\bf a}:\Om\ra\R^{d(d+1)/2}$  from $\Om$ to the space of symmetric $d\times d$ matrices  satisfies the quadratic form inequality  
\begin{equation} \label{B1}
\la I_d \le {\bf a}(\om) \le \La I_d, \ \ \ \ \ \om\in\Om,
\end{equation}
where $I_d$ is the identity matrix in $d$ dimensions and $\La, \la$
are positive constants.

It is well known \cite{k,pv,zko} that if the translation operators $\tau_x, \ x\in\Z^d$, are ergodic  on $\Om$ then solutions to the random equation (\ref{A1}) converge to solutions of a constant coefficient equation under suitable scaling. Thus suppose $f:\R^d\ra\R$ is a $C^\infty$ function with compact support and for $\ve$ satisfying $0<\ve\le 1$ let $u_\ve(x,\eta,\om)$ be the solution to (\ref{A1}) with $h(x)=\ve^2f(\ve x), \ x\in\Z^d$.  Then $u_\ve(x/\ve,\ve^2\eta,\om)$ converges with probability $1$ as $\ve\ra 0$ to a function $u_{\rm hom}(x,\eta), \ x\in\R^d$, which is the solution to the constant coefficient elliptic PDE
\be \label{C1}
\eta u_{\rm hom}(x,\eta)-\na {\bf a}_{\rm hom}\na u_{\rm hom}(x,\eta)=f(x), \quad x\in\R^d,
\ee
where the $d\times d$ symmetric matrix ${\bf a}_{\rm hom}$ satisfies the quadratic form inequality (\ref{B1}).   This homogenization result can be viewed as a kind of central limit theorem, and our goal in \cite{cs} was to show that the result can be strengthened for certain probability spaces $(\Om,\mathcal{F},P)$. In particular, we extended a result of Yurinskii\ \cite{y} which gives a  rate of convergence in homogenization:
 \begin{theorem}
 Let  $f:\R^d\ra\R$ be a $C^\infty$ function of compact support,  $u_\ve(x,\eta,\om)$  the corresponding solution to (\ref{A1}) with $h(x)=\ve^2f(\ve x), \ x\in\Z^d,$ and $u_{\rm hom}(x,\eta), \ x\in\R^d,$ the solution to (\ref{C1}).  Then for certain strong mixing environments $(\Om,\mathcal{F},P)$-see discussion below-  there is a constant $\al>0$ depending only on $d,\La/\la$ and a constant $C$ independent of $\ve$ such that
 \be \label{D1}
 \sup_{x\in\ve\Z^d}\langle \ |u_\ve(x/\ve,\ve^2\eta,\cdot)-u_{\rm hom}(x,\eta)|^2 \ \rangle \ \le C\ve^{\al}, \quad {\rm for \ } 0<\ve\le 1.
 \ee
 \end{theorem}
 It is evident that any environment $\Om$  can be considered to be a set of fields $\om:\Z^d\ra\R^n$ with $n\le d(d+1)/2$, where the translation operators $\tau_x, \ x\in\Z^d$, act as $\tau_x\om(z)=\om(x+z), \ z\in\Z^d,$ and  $\mathbf{a}(\om)=\tilde{\mathbf{a}}(\om(0))$ for some function $\tilde{\mathbf{a}}:\R^n\ra\R^{d(d+1)/2}$.
Yurinskii's assumption on $(\Om,\mathcal{F},P)$ is a quantitative strong mixing condition. Thus let $\chi(\cdot)$ be a positive decreasing function on $\R^+$ such that $\lim_{q\ra\infty}\chi(q)=0$. The quantitative strong mixing condition is given in terms of  the function $\chi(\cdot)$ as follows:  For any subsets $A,B$ of $\Z^d$ and events $\Gamma_A, \ \Gamma_B\subset\Om,$ which depend respectively only on variables $\om(x), \ x\in A,$ and $\om(y), \ y\in B$, then
 \be \label{E1}
 |P(\Gamma_A\cap \Gamma_B)-P(\Gamma_A)P(\Gamma_B)| \ \le  \chi(\inf_{x\in A,y\in B} |x-y| \ ) \ .
 \ee
 In the proof of (\ref{D1}) he requires the function $\chi(\cdot)$ to have power law decay i.e. $\lim_{q\ra\infty} q^\al\chi(q)=0$ for some $\al>0$.   Evidently (\ref{E1}) trivially holds if the $\om(x), \ x\in\Z^d,$ are independent variables. Recently Caffarelli and Souganidis \cite {cafs} have obtained rates of convergence results in homogenization of fully nonlinear PDE under the quantitative strong mixing condition (\ref{E1}). In their case the function $\chi(q)$ is assumed to decay logarithmically in $q$ to $0$, and correspondingly the rate of convergence in homogenization that is obtained is also logarithmic in $\ve$. In their methodology a stronger assumption on the function $\chi(\cdot)$, for example power law decay, does not  yield a stronger  rate of convergence in homogenization.
 
 In \cite{ns2}  Naddaf and Spencer obtained rate of convergence results for homogenization under a different quantitative strong mixing assumption  than (\ref{E1}). Their assumption is that a   Poincar\'{e} inequality holds for the random environment.  Specifically,  consider the measure space   $(\tilde{\Om},\tilde{\mathcal{F}})$ of   vector fields  $\tilde{\om}:\Z^d\ra\R^k$, where  $\tilde{\mathcal{F}}$ is the minimal Borel algebra such that each $\tilde{\om}(x):\tilde{\Om}\ra\R^k$ is Borel measurable,  $x\in\Z^d$.  A probability measure $\tilde{P}$ on $(\tilde{\Om},\tilde{\mathcal{F}})$ satisfies a Poincar\'{e} inequality if   there is a constant $m>0$ such that
 \be \label{F1}
 {\rm Var}[ G(\cdot)] \ \le \ \frac{1}{m^2} \langle \ \|d_{\tilde{\om}} G(\cdot)\|^2 \ \rangle  \quad 
 {\rm for \  all \  } C^1  \ {\rm functions} \  G:\tilde{\Om}\ra\C,
 \ee
where $d_{\tilde{\om}}G(y,\tilde{\om})=\pa G(\tilde{\om})/\pa \tilde{\om}(y), \ y\in\Z^d,$ is the gradient of $G(\cdot)$.  In \cite{ns2} it is assumed that $\tilde{P}$ is translation invariant i.e. the  translation operators $\tau_x, \ x\in\Z^d$, acting by $\tau_x\tilde{\om}(z)=\tilde{\om}(x+z), \ z\in\Z^d,$ are measure preserving, and that the  Poincar\'{e} inequality (\ref{F1}) holds. Rate of convergence results are then obtained provided  $\mathbf{a}(\om)=\tilde{\mathbf{a}}(\tilde{\om}(0))$ in (\ref{A1}), where the function $\tilde{\mathbf{a}}:\R^k\ra\R^{d(d+1)/2}$ is $C^1$  and has bounded derivative, in addition to satisfying (\ref{B1}). In recent work  Gloria and Otto \cite{go1,go2} have developed much further the methodology of Naddaf and Spencer, obtaining optimal rates of convergence in homogenization under the assumption that the environment satisfies a {\it weak} Poincar\'{e} inequality.   This weak  Poincar\'{e} inequality holds for an environment in which the variables $\mathbf{a}(\tau_x\om), \ x\in \Z^d,$ are independent, whereas the inequality (\ref{F1}) in general does not. 

In this paper we shall obtain rate of convergence and related results for homogenization  in environments defined by $\mathbf{a}(\om)=\tilde{\mathbf{a}}(\om(0))$ where $\om:\Z^d\ra\R^n$ is a convolution  $\om(\cdot)=h*\tilde{\om}(\cdot), \ \tilde{\om}\in\tilde{\Om}$. The function $h:\Z^d\ra \R^n\otimes\R^k$ from $\Z^d$ to $n\times k$ matrices is assumed to be $q$ integrable for some $q<2$, and the probability space  $(\tilde{\Om},\tilde{\mathcal{F}},\tilde{P})$ to satisfy the Poincar\'{e} inequality  (\ref{F1}).  Unlike environments which satisfy a Poincar\'{e} inequality, these environments  $\om(\cdot)$ can easily be shown to have long range correlations. In particular, consider the massive field theory environment studied in \cite{cs} consisting of fields $\phi:\Z^d\ra\R$ with measure $\tilde{P}$ formally given by  
\be \label{G1}
\exp \left[ - \sum_{x\in \Z^d} V\left( \na\phi(x)\right)+\frac{1}{2}m^2\phi(x)^2 \right] \prod_{x\in \Z^d} d\phi(x)/{\rm normalization},
\ee
where $V:\R^d\ra\R$ is a uniformly convex function. Then $(\tilde{\Om},\tilde{\mathcal{F}},\tilde{P})$ with measure (\ref{G1}) satisfies the inequality (\ref{F1}). In the Gaussian case when $V(\cdot)$ is quadratic one has that the correlation function $\langle \ \phi(x)\phi(0) \ \rangle=G_{m^2}(x), \ x\in\Z^d,$ where the Green's function $G_\eta(\cdot)$  is the solution to
\be \label{H1}
\eta G_\eta(x)+\na^* V''\na G_\eta(x) \ = \ \del(x), \quad x\in\Z^d.
\ee
Hence the correlation function $\langle \ \phi(x)\phi(0)  \ \rangle$ decays exponentially in $|x|$ as $|x|\ra\infty$. 
Taking $\om(\cdot)=h*\phi(\cdot)$ for some $h\in L^q(\Z^d)$ we have that
\be \label{I1}
\langle \ \om(x)\om(0) \ \rangle \ = \ \sum_{y,y'\in\Z^d} h(x-y)h(-y')G_{m^2}(y-y') \ ,
\ee
and so if $1\le q\le 2$ then $\langle \om(0)^2\rangle<\infty$.  Choosing now $h(z)=1/[1+|z|^{d/2+\ve}],  \ z\in\Z^d,$ for any $\ve>0$ we see from (\ref{I1}) that $\langle \ \om(x)\om(0) \ \rangle \simeq       |x|^{-2\ve}$ as $|x|\ra\infty$. 

The limit as $m\ra 0$ of the measure (\ref{G1}) is a probability measure $\tilde{P}$ on gradient fields $\tilde{\om}:\Z^d\ra\R^d$, where formally $\tilde{\om}(x)=\na\phi(x), \ x\in\Z^d$, a result first shown by Funaki and Spohn \cite{fs}.  This massless field theory measure satisfies a Poincar\'{e} inequality (\ref{F1})  for all $d\ge 1$. In the case $d=1$ the measure has a simple structure since  then the variables $\tilde{\om}(x), \ x\in\Z$, are i.i.d.    For $d\ge 3$ the gradient field theory measure induces a measure on fields $\phi:\Z^d\ra\R$ which is simply the limit of the measures (\ref{G1})  as $m\ra 0$.   For $d=1,2$ the $m\ra 0$ limit of the measures (\ref{G1}) on fields  $\phi:\Z^d\ra\R$ does not exist.   If $d\ge 3$ then $\langle \ \phi(x)\phi(0) \ \rangle \simeq |x|^{-(d-2)}$ as $|x|\ra\infty$ for the massless field theory.    Observe now that 
\be \label{J1}
\phi(x) \ = \ \sum_{y\in\Z^d} [\na G_0(x-y)]^T\na\phi(y) \ = \ h*\tilde{\om}(x), \quad x\in\Z^d,
\ee
where $G_0(\cdot)$ is the Green's function for (\ref{H1}) with $\eta=0, V''=I_d$. Since $h:\Z^d\ra\R^d$ in (\ref{J1})  is $q$ integrable for any $q>d/(d-1)$, the environment  of massless fields $\phi:\Z^d\ra\R$ with $d\ge 3$ satisfies the condition $q<2$ of the previous paragraph. 

Rather than attempt to formulate a general theorem for environments $\om=h*\tilde{\om}$ where 
 $(\tilde{\Om},\tilde{\mathcal{F}},\tilde{P})$  satisfies the Poincar\'{e} inequality (\ref{F1}),  we shall only rigorously prove  that the results obtained in \cite{cs}  hold  for massless fields $\phi:\Z^d\ra\R$ with $d\ge 3$.  In $\S2$ we indicate the generality of our argument by showing that  the proof of Proposition 5.3 of \cite{cs} formally extends to environments  $\om=h*\tilde{\om}$. In $\S3$ we implement the method of $\S2$ to prove the following theorem for massless fields:
 
\begin{theorem} Let $\tilde{{\bf a}}:\R\ra\R^{d(d+1)/2}$  be a $C^1$ function   on $\R$ with values in the space of symmetric $d\times d$ matrices, which satisfies the quadratic form inequality  (\ref{B1}) and has bounded first derivative $D\tilde{{\bf a}}(\cdot)$ so  $\|D\tilde{{\bf a}}(\cdot)\|_{\infty}<\infty$.  For $d\ge 3$ let $(\Om, \mathcal{F}, P)$ be the probability space of massless  fields $\phi(\cdot)$ determined by the limit of the uniformly convex measures (\ref{G1}) as $m\ra 0$, and set ${\bf a}(\cdot)$ in (\ref{A1}) to be ${\bf a}(\phi)=\tilde{{\bf a}}(\phi(0)), \ \phi\in\Om$.   Then Theorem 1.1 holds for the probability space  $(\Om, \mathcal{F}, P)$. Let $G_{\mathbf{a},\eta}(x), \ x\in\Z^d,$ be the averaged Green's function  for the random equation (\ref{A1}) and  $G_{{\bf a}_{\rm hom},\eta}(x), \ x\in\R^d,$ the Green's function for the homogenized equation (\ref{C1}). Then there are constants $\al,\ga>0$ depending only on $d$ and the ratio $\La/\la$ of the constants $\la,\La$ of (\ref{B1}), and a constant $C$ depending only on $\|D\tilde{{\bf a}}(\cdot)\|_{\infty},\La/\la,d$ such that 
\be \label{K1}
|G_{{\bf a},\eta}(x)-G_{{\bf a}_{\rm hom},\eta}(x)|  \le  \frac{C}{\La(|x|+1)^{d-2+\alpha}} e^{-\gamma\sqrt{\eta/\La} |x|}, \    \ x\in\Z^d-\{0\},  
\ee
\be\label{L1}
 |\na G_{{\bf a},\eta}(x)-\na G_{{\bf a}_{\rm hom},\eta}(x)|  \le \frac{C}{\La(|x|+1)^{d-1+\alpha}} e^{-\gamma\sqrt{\eta/\La} |x|},   \    \ x\in\Z^d-\{0\},  
 \ee
 \be \label{M1}
 |\na\na G_{{\bf a},\eta}(x)-\na\na G_{{\bf a}_{\rm hom},\eta}(x)|  \le \frac{C}{\La(|x|+1)^{d+\alpha}} e^{-\gamma\sqrt{\eta/\La} |x|}  \    \ x\in\Z^d-\{0\},
\ee
provided $0<\eta\le \La$. 
\end{theorem}

\vspace{.1in}

\section{Variance Estimate on the Solution to a PDE on $\Om$}
Following $\S5$ of \cite{cs} we consider the solution $\Phi(\xi,\eta,\om)$ to the equation
\be \label{A2}
\eta\Phi(\xi,\eta,\om)+\pa_\xi^*{\bf a}(\om)\pa_\xi\Phi(\xi,\eta,\om)=-\pa^*_\xi {\bf a}(\om), \quad \eta>0, \ \xi\in \R^d, \ \om\in\Om,
\ee
and let $P$ denote the projection orthogonal to the constant function. Then our generalization of Proposition 5.3 of \cite{cs}  is as follows:
\begin{proposition}
Suppose ${\bf a}(\cdot)$ in (\ref{A2}) is given by   $\mathbf{a}(\om)=\tilde{\mathbf{a}}(\om(0))$ where $\tilde{\mathbf{a}}:\R^n\ra\R^{d(d+1)/2}$ is a $C^1$ $d\times d$ symmetric matrix valued function  satisfying the quadratic form inequality (\ref{B1}) and $\|D\tilde{a}(\cdot)\|_\infty<\infty$.  The random field $\om:\Z^d\ra\R^n$ is a convolution  $\om(\cdot)=h*\tilde{\om}(\cdot)$ of an $n\times k$ matrix valued function $h:\Z^d\ra\R^n\otimes\R^k$ and a random field $\tilde{\om}:\Z^d\ra\R^k$.  The function $h$ is assumed to be $p_0$ summable for some $p_0$ with $1\le p_0<2$ and the probability space $(\tilde{\Om},\tilde{\mathcal{F}},\tilde{P})$ of the fields $\tilde{\om}:\Z^d\ra\R^k$ to satisfy the Poincar\'{e} inequality (\ref{F1}). Then there exists $p_1$ depending only on $d,\La/\la,p_0$ and satisfying $1<p_1< 2$, such that for $g\in L^p(\Z^d,\C^d\otimes\C^d)$ with $1\le p\le p_1$ and $v\in\C^d$,
\be \label{B2}
\|P\sum_{x\in\Z^d} g(x)\pa_\xi\Phi(\xi,\eta,\tau_x\cdot)v\|   \ \le \ \frac{C\|D\tilde{a}(\cdot)\|_\infty|v|}{m\La}\|h\|_{p_0} \|g\|_p \ ,
\ee
where $C$ depends only on $d,n,k,\La/\la,p_0$. 
\end{proposition}
\begin{proof}
From (\ref{F1}) we have that
\be \label{C2}
\|P\sum_{x\in\Z^d} g(x)\pa_\xi\Phi(\xi,\eta,\tau_x\cdot)v\|^2   \ \le \ 
\frac{1}{m^2}\sum_{z\in \Z^d} \| \ \frac{\pa}{\pa \tilde{\om}(z)}  \ \sum_{x\in\Z^d} g(x)\pa_\xi\Phi(\xi,\eta,\tau_x\cdot)v\|^2 \ . 
\ee
From the chain rule we see that
\be \label{D2}
\frac{\pa}{\pa \tilde{\om}(z)} \pa_\xi\Phi(\xi,\eta,\tau_x\cdot)v \ = \  \sum_{y\in\Z^d} 
\left[\frac{\pa}{\pa \om(y)} \pa_\xi\Phi(\xi,\eta,\tau_x\cdot)v\right] h(y-z) \ .
\ee
Hence using the translation invariance of the probability measure $\tilde{P}$ on $\tilde{\Om}$ we conclude from (\ref{C2}), (\ref{D2}) that
\be \label{E2}
\|P\sum_{x\in\Z^d} g(x)\pa_\xi\Phi(\xi,\eta,\tau_x\cdot)v\|^2   \ \le \ 
\frac{1}{m^2}\sum_{z\in \Z^d} \left\|   \ \sum_{x\in\Z^d} g(x)\sum_{y\in\Z^d} \left[ \tau_{-z}
\frac{\pa}{\pa \om(y)} \pa_\xi\Phi(\xi,\eta,\tau_x\cdot)v\right]h(y-z) \right\|^2 \ . 
\ee
We define a function $u:\Z^d\times\Om\ra\C^k$ by 
\be \label{F2}
u(z,\om) \ = \ e^{-iz\cdot\xi}\sum_{y\in\Z^d} \left[
 d_\om\Phi(y,\xi,\eta,\tau_z\om)v\right]h(y+z) \ ,
\ee
where $ d_\om\Phi(\cdot,\xi,\eta,\om)v:\Z^d\ra\C^n$ is the gradient of $\Phi(\xi,\eta,\om)v$ with respect to $\om\in\Om$.
Observe now  that
\be \label{G2}
\na u(x-z,\om) \ = \ e^{i(z-x).\xi}\sum_{y\in\Z^d} \left[ \tau_{-z}
\frac{\pa}{\pa \om(y)} \pa_\xi\Phi(\xi,\eta,\tau_x\om)v\right]h(y-z) \ ,
\ee
whence (\ref{E2}) becomes
\be \label{H2}
\|P\sum_{x\in\Z^d} g(x)\pa_\xi\Phi(\xi,\eta,\tau_x\cdot)v\|^2   \ \le \ 
\frac{1}{m^2}\sum_{z\in \Z^d} \|   \ \sum_{x\in\Z^d} g(x)e^{i(x-z)\cdot\xi}\na u(x-z,\cdot) \ \|^2 \ .
\ee

Next we take the gradient of  equation (\ref{A2}) with respect to $\om(\cdot)$.  Using the notation of \cite{cs} we have that
\begin{multline} \label{I2}
\eta \   d_\om\Phi(y,\xi,\eta,\om)v+D_\xi^*\tilde{{\bf a}}(\om(0))D_\xi  \  d_\om\Phi(y,\xi,\eta,\om)v \\
 = \ -D_\xi^*[ \ \del(y) D\tilde{{\bf a}}(\om(0))\{v+\pa_\xi \Phi(\xi,\eta,\om)v\} ] \quad {\rm for \ } y\in\Z^d,\om\in\Om.
\end{multline}
Evidently (\ref{I2}) holds with $\om\in\Om$ replaced by $\tau_z\om$ for any $z\in\Z^d$. We now multiply (\ref{I2}) with $\tau_z\om$ in place of $\om$ on the right by $e^{-iz\cdot\xi}h(y+z)$ and sum with respect to $y\in\Z^d$. It then follows from (\ref{F2}) that 
\be \label{J2}
\eta \ u(z,\om)+\na^*\tilde{{\bf a}}(\om(z))\na u(z,\om) \ = \ -\na^*f(z,\om) \ ,
\ee
where the function $f:\Z^d\times\Om\ra\C^d\otimes\C^k$ is given by the formula
\be \label{K2}
f(z,\om) \ = \  D\tilde{{\bf a}}(\om(z))\{v+\pa_\xi \Phi(\xi,\eta,\tau_z\om)v\} e^{-iz\cdot\xi} h(z) \ .
\ee
Now $\pa_\xi \Phi(\xi,\eta,\cdot)v\in L^2(\Om,\C^d)$ and  $\|\pa_\xi \Phi(\xi,\eta,\cdot)v\|_2\le \La|v|/\la$. Hence if $h\in L^2(\Z^d,\R^n\otimes\R^k)$ then the function $f$ is in $L^2(\Z^d\times\Om,
\C^d\otimes\C^k)$ and $\|f\|_2\le \|D\tilde{{\bf a}}(\cdot)\|_\infty(1+\La/\la)|v|\|h\|_2$.  We see  from (\ref{J2}) that if  $f\in L^2(\Z^d\times\Om,
\C^d\otimes\C^k)$ then $\na u$ is in $L^2(\Z^d\times\Om,\C^d\otimes\C^k)$ and $\|\na u\|_2\le \|f\|_2/\la$. It follows then from (\ref{H2}) and Young's inequality that  (\ref{B2}) holds with $p_0=2$ and $p=1$.  

To prove the inequality for some $p>1$ we use Meyer's theorem \cite{m}. Thus for any $1<q<\infty$ we  consider the function $f$ as a mapping $f:\Z^d\ra  L^2(\Om,\C^d\otimes\C^k)$  with norm defined by
\be \label{L2}
\|f\|_q^q \ = \  \sum_{z\in\Z^d} \|f(z,\cdot)\|_2^q \ ,
\ee
where $\|f(z,\cdot)\|_2$ is the norm of $f(z,\cdot)\in L^2(\Om,\C^d\otimes\C^k)$. Noting that the Calderon-Zygmund theorem applies to functions with range in a Hilbert space \cite{stein}, we conclude that there exists $q_0$ depending only on $d,\La/\la$ with $1<q_0<2$ such that if  $\|f\|_{q_0}<\infty$ then $\|\na u\|_q\le2\|f\|_q/\la$ for $q_0\le q\le 2$. If $h$ is $p_0$ integrable with $p_0<2$ we can take $\max[p_0,q_0]=q_1\le  q\le 2$. It follows  again from (\ref{H2}) and Young's inequality that (\ref{B2}) holds with $p_1=2q_1/(3q_1-2)$.  
\end{proof}

\vspace{.1in}

\section{Proof of Theorem 1.2}
It will be sufficient for us to establish the conclusion of Lemma 5.1 of \cite{cs} for the massless field theory environment $(\Om,\mathcal{F},P)$ of Theorem 1.2. Using the notation of \cite{cs} we have the following:
\begin{lem}
Let $(\Om,\mathcal{F},P)$ be an environment of massless fields $\phi:\Z^d\ra\R$ with $d\ge 3$, and $\tilde{\mathbf{a}}:\R\ra\R^{d(d+1)/2}$ be  as in the statement of Theorem 1.2. Set  $\mathbf{a}(\phi)=\tilde{\mathbf{a}}(\phi(0)), \ \phi\in\Om$.  Then there exists $p_0(\La/\la)$ with $1<p_0(\La/\la)<2$ depending only on $d$ and $\La/\la$, and a constant $C$ depending only on $d$ such that
\be \label{A3}
\|T_{r,\eta}\|_{p,\infty} \ \le \ \frac{C r\|D\tilde{\mathbf{a}}(\cdot)\|_\infty}{\la\La}(1-\la/\La)^{r/2} \quad {\rm \ for \ } 1\le p\le p_0(\La/\la) \ .
\ee
\end{lem}
\begin{proof}
It will be sufficient for us to bound $\|T_{r,\eta}g\|_\infty$ in terms of $\|g\|_p$ for $g:\Z^d\ra\C^d\otimes\C^d$ of finite support. Let $Q$ be a cube in $\Z^d$ containing the support of the function $g(\cdot)$ and  $(\Om_Q,\mathcal{F}_Q,P_{Q,m})$ be the probability space of periodic functions $\phi:Q\ra\R$ with measure
\be \label{B3}
\exp \left[ - \sum_{x\in Q} V\left( \na\phi(x)\right)+\frac{1}{2}m^2\phi(x)^2 \right] \prod_{x\in  Q} d\phi(x)/{\rm normalization},
\ee
where we assume $m>0$ and $V:\R^d\ra\R$ is $C^2$ with  $\mathbf{a}(\cdot)=V''(\cdot)$ satisfying the quadratic form inequality (\ref{B1}).  We denote by $\tilde{\Om}_Q$ the space  of periodic fields $\tilde{\om}:Q\ra\R^d$ and let $F:\tilde{\Om}_Q\times\Om_Q\ra\C$ be a $C^1$ function  which for some constants $A,B$  satisfies the inequality
\be \label{C3}
|F(\tilde{\om},\phi)|+|d_{\tilde{\om}}F(y,\tilde{\om},\phi)|+|d_\phi F(y,\tilde{\om},\phi)|  \ \le \  A\exp[B\{\|\tilde{\om}\|_2+\|\phi\|_2\}] \ ,  \qquad  y\in Q, \ \tilde{\om}\in\tilde{\Om}_Q, \ \phi\in\Om_Q \ .
\ee
Letting $\langle\cdot\rangle_{Q,m}$ denote expectation with respect to the measure (\ref{B3}) we see from  the Brascamp-Lieb inequality \cite{bl} that the Poincar\'{e} inequality
\be \label{D3}
{\rm Var}_{Q,m}[F(\na\phi,\phi)] \ \le \  \frac{2}{\la}\langle \ \|d_{\tilde{\om}} F(\na\phi,\phi)\|^2
 \ \rangle_{Q,m} +\frac{2}{m^2}\langle \ \|d_\phi F(\na\phi,\phi)\|^2
 \ \rangle_{Q,m} \ 
\ee
holds.  We shall show using (\ref{D3}) that $\|T_{r,\eta}g\|_\infty$ is bounded in terms of $\|g\|_p$ if the environment is the probability space  $(\Om_Q,\mathcal{F}_Q,P_{Q,m})$. The result will then follow by taking first $Q\ra\Z^d$ and then $m\ra 0$. 

Let us suppose that the cube $Q$ is centered at the origin in $\Z^d$ with side of length $L$, where $L$ is an even integer. Let  $G_\eta:\Z^d\ra\R$ be the solution to  (\ref{H1}) with $V''(\cdot)=I_d$ and $G_{\eta,Q}:Q\ra\R$ the corresponding Green's function for the periodic lattice $Q$, so
\be \label{E3}
G_{\eta,Q}(x) \ = \ \sum_{n\in\Z^d} G_\eta(x+Ln) \ , \quad x\in Q.
\ee
Then any periodic function $\phi:Q\ra\R$ can be written as
\be \label{F3}
\phi(x) \ = \ \sum_{y\in Q}[ \na G_{\eta,Q}(x-y)]^T\na\phi(y)+  \sum_{y\in Q} \eta G_{\eta,Q}(x-y)\phi(y) \ , \quad x\in Q.
\ee
Taking $\eta=1/L^2$ in (\ref{F3}) we have a representation
\begin{multline} \label{G3}
\phi(\cdot) \ = \ h^T_Q*\tilde{\om}(\cdot)+k_Q*\phi(\cdot), \quad {\rm where \ } \|h_Q\|_q\le \  C_q \ {\rm for \ } q>d/(d-1), \\
{\rm and \ } \|k_Q\|_q \ \le \ C_q/\min[L^{d(1-1/q)},L^2] \quad {\rm for \ } q\ge 1 \ {\rm and \ } q\ne d/(d-2) \ .
\end{multline}
In (\ref{G3}) the vector  $h_Q=[h_{Q,1},...,h_{Q,d}]$ is a column vector,  the operation $*$ denotes convolution on the periodic lattice $Q$, and $C_q$ is a constant depending only on $q,d$. 

We first prove (\ref{A3}) when $r=1$. For the environment $(\Om_Q,\mathcal{F}_Q,P_{Q,m})$ we have from (\ref{G3})  that
\be \label{H3}
T_{1,\eta} g(\xi,\phi) \ = \  \sum_{x\in Q} g(x)e^{-ix\cdot\xi}P\tilde{\mathbf{b}}\big(h^T_Q*\tilde{\om}(x)+k_Q*\phi(x)\big) \ .
\ee
Let $\mathcal{H}_m(Q)$ be the Hilbert space of functions $f:\Om_Q\ra\C^d$ which are square integrable with respect to the measure $P_{Q,m}$.  It follows  from (\ref{D3}) that if $v\in\C^d$   the norm of $T_{1,\eta} g(\xi,\cdot)v\in\mathcal{H}_m(Q)$ is bounded as
\be \label{I3}
\|T_{1,\eta} g(\xi,\cdot)v\|^2 \ \le \  \frac{2}{\la}\sum_{z\in Q} \sum_{j=1}^d\|\sum_{x\in Q}g(x)h_{Q,j}(x-z)D\tilde{\mathbf{b}}(\phi(x))v\|^2+
\frac{2}{m^2}\sum_{z\in Q} \|\sum_{x\in Q} g(x)k_Q(x-z)D\tilde{\mathbf{b}}(\phi(x))v\|^2 \ .
\ee
Since $d\ge 3$ we can choose $q$ such that $d/(d-1)<q<2$ and $q\ne d/(d-2)$.  It then follows  from (\ref{G3}), (\ref{I3})  that for $p=2q/(3q-2)>1$ 
\be \label{J3}
\|T_{1,\eta} g(\xi,\cdot)v\|^2 \ \le \  C_q\|g\|^2_p\|D\tilde{\mathbf{b}}(\cdot)\|^2_\infty|v|^2\left[\frac{1}{\la}+\frac{1}{m^2 L^{a(q)}}\right] \ ,
\ee
where $a(q)=2\min[d(1-1/q),2]$. Let $(\Om,\mathcal{F},P_m)$ be the probability space of fields $\phi:\Z^d\ra\R$ with measure $P_m$ given by (\ref{G1}).  Proposition 5.1 of \cite{cs} enables us to take the limit of (\ref{J3}) as $Q\ra\Z^d$ to obtain the inequality
 \be \label{K3}
\|T_{1,\eta} g(\xi,\cdot)v\|^2 \ \le \  C_q\|g\|^2_p\|D\tilde{\mathbf{b}}(\cdot)\|^2_\infty|v|^2\big/\la \ 
\ee
for the environment $(\Om,\mathcal{F},P_m)$. Finally Proposition 6.1 of \cite{cs} enables us to take the limit of (\ref{K3}) as $m\ra 0$ provided $d\ge 3$.  We have proved (\ref{A3}) when $r=1$. 

To prove the result for $r>1$ we consider the environment $(\Om_Q,\mathcal{F}_Q,P_{Q,m})$  and  write as in \cite{cs}
\be \label{L3}
T_{r,\eta}g(\xi,\phi(\cdot))v \ = \ P\sum_{x\in Q} g(x) e^{-ix.\xi}\tilde{\mathbf{b}}(\phi(x))\pa_\xi F_r(\tau_x\phi(\cdot)) \ ,
\quad \phi(\cdot)\in\Om_Q,
\ee
where the functions $F_r(\phi(\cdot))$ are defined inductively by
\begin{eqnarray} \label{M3}
\frac{\eta}{\La}F_r(\phi(\cdot))+\pa_\xi^*\pa_\xi F_r(\phi(\cdot))&=&P\pa^*_\xi [\tilde{{\bf b}}(\phi(0))\pa_\xi F_{r-1}(\phi(\cdot))],  \ r> 2,  \\
\frac{\eta}{\La}F_2(\phi(\cdot))+\pa_\xi^*\pa_\xi F_2(\phi(\cdot))&=&P\pa^*_\xi [\tilde{{\bf b}}(\phi(0))v]  \ .
\nonumber
\end{eqnarray}
It is easy to see that $\pa_\xi F_r\in\mathcal{H}_m(Q)$ and
\be \label{N3}
\|\pa_\xi F_r\| \ \le \ (1-\la/\La)^{r-1}|v|  \quad {\rm for \ } r\ge 2.
\ee
Using the representation (\ref{G3}) for  $\phi(\cdot)$ we can consider the $F_r, \ r\ge 2,$ defined by (\ref{M3}) as  functions of $\tilde{\om}(\cdot)$ and $\phi(\cdot)$, which we denote by  $F_r(\tilde{\om},\phi)$.    Observe now that for $1\le j\le d$,
\begin{multline} \label{O3}
\frac{\pa}{\pa\tilde{\om}_j(z)}\sum_{x\in Q} g(x) e^{-ix.\xi}\tilde{\mathbf{b}}(h_Q^T*\tilde{\om}(x)+k_Q*\phi(x))\pa_\xi F_r(\tau_x\tilde{\om},\tau_x\phi)  \ = \\ 
\sum_{x\in Q} g(x) e^{-ix.\xi}h_{Q,j}(x-z)D\tilde{\mathbf{b}}(\phi(x))\pa_\xi F_r(\tau_x\phi(\cdot)) 
+  \sum_{x\in Q} g(x) e^{-ix.\xi}\tilde{\mathbf{b}}(\phi(x))\frac{\pa}{\pa\tilde{\om}_j(z)}\pa_\xi F_r(\tau_x\tilde{\om},\tau_x\phi)  \ .
\end{multline}
Let $u_{r,j}:Q\times\Om_Q\ra\C$ be given by the formula
\be \label{P3}
u_{r,j}(z,\phi(\cdot)) \ = \ e^{-iz\cdot\xi}\sum_{y\in Q} d_\phi F_r(y,\tau_z\phi(\cdot))h_{Q,j}(y+z) \ .
\ee
Then as in (\ref{D2}), (\ref{G2}) we have that
\begin{multline} \label{Q3}
\na u_{r,j}(x-z,\phi(\cdot)) \ = \ e^{i(z-x)\cdot\xi}\sum_{y\in Q} \tau_{-z}\frac{\pa}{\pa\phi(y)}\pa_\xi F_r(\tau_x\phi(\cdot))h_{Q,j}(y-z)  \\
= \ e^{i(z-x)\cdot\xi} \tau_{-z}\frac{\pa}{\pa\tilde{\om}_j(z)}\pa_\xi F_r(\tau_x\tilde{\om},\tau_x\phi)  \ .
\end{multline} 
Similarly to (\ref{J2}), (\ref{K2}) we see that $u_{r,j}(z,\phi(\cdot))$ satisfies the equation 
\be \label{R3}
\frac{\eta}{\La} \ u_{r,j}(z,\phi(\cdot))+\na^*\na u_{r,j}(z,\phi(\cdot)) \ = \ P\na^*f_{r,j}(z,\phi(\cdot)) \ ,
\ee
where the function $f_{r,j}:Q\times\Om_Q\ra\C^d$ is given by the formula
\begin{multline} \label{S3}
f_{2,j}(z,\phi) \ = \  D\tilde{{\bf b}}(\phi(z))v e^{-iz\cdot\xi} h_{Q,j}(z) \ , \\
f_{r,j}(z,\phi) \ = \  D\tilde{{\bf b}}(\phi(z))\pa_\xi F_{r-1}(\tau_z\phi(\cdot)) e^{-iz\cdot\xi} h_{Q,j}(z)+\tilde{{\bf b}}(\phi(z))\na u_{r-1,j}(z,\phi(\cdot)), \ r>2.
\end{multline}

From (\ref{O3}), (\ref{Q3}) we see that
\begin{multline} \label{T3}
\frac{1}{2}\big\|\frac{\pa}{\pa\tilde{\om}_j(z)}\sum_{x\in Q} g(x) e^{-ix.\xi}\tilde{\mathbf{b}}(h_Q^T*\tilde{\om}(x)+k_Q*\phi(x))\pa_\xi F_r(\tau_x\tilde{\om},\tau_x\phi)\big\|^2 \ \le  \\
\big\|\sum_{x\in Q} g(x) e^{-ix.\xi}h_{Q,j}(x-z)D\tilde{\mathbf{b}}(\phi(x-z))\pa_\xi F_r(\tau_{x-z}\phi(\cdot))\big\|^2 +\big\|\sum_{x\in Q} g(x)e^{-iz\cdot\xi} \tilde{\mathbf{b}}(\phi(x-z))\na u_{r,j}(x-z,\phi(\cdot))\big\|^2 \ .
\end{multline}
Observe now from (\ref{N3}) and Young's inequality for functions with values in a Hilbert space that
\begin{multline} \label{U3}
\sum_{z\in Q} \big\|\sum_{x\in Q} g(x) e^{-ix.\xi}h_{Q,j}(x-z)D\tilde{\mathbf{b}}(\phi(x-z))\pa_\xi F_r(\tau_{x-z}\phi(\cdot))\big\|^2  \\
\le \ C\left[\|D\tilde{\mathbf{b}}(\cdot)\|_\infty \|g\|_p\|h_{Q,j}\|_q (1-\la/\La)^{r-1}|v|\right]^2 \ ,
\end{multline}
where $p=2q/(3q-2)$ with $1\le q\le 2$ and $C$ depends only on $d$. We can bound the second term on the RHS of (\ref{T3}) similarly. Thus let $L^q(Q,\mathcal{H}_m(Q))$ be the Banach space of functions $f:Q\ra \mathcal{H}_m(Q)$ with norm
\be \label{V3}
\|f\|_q^q \ =  \ \sum_{x\in Q} \|f(x)\|^q \ .
\ee
From (\ref{S3}) it follows that
\begin{multline} \label{W3}
\|f_{2,j}\|_q \ \le \  C\|D\tilde{\mathbf{b}}(\cdot)\|_\infty\|h_{Q,j}\|_q |v| \ , \\
\|f_{r,j}\|_q \ \le \ C\|D\tilde{\mathbf{b}}(\cdot)\|_\infty\|h_{Q,j}\|_q(1-\la/\La)^{r-2} |v|+
(1-\la/\La)\|\na u_{r-1,j}\|_q \  \quad {\rm if \ } r>2,
\end{multline}
where $C$ depends only on $d$. We see from the Hilbert space version of  the Calderon-Zygmund theorem  (see \cite{stein} page 45) applied to (\ref{R3}) that
for $q>1$ there is a constant $\del(q)\ge 0$ such that
\be \label{X3}
\|\na u_{r,j}\|_q \ \le \ [1+\del(q)]\|f_{r,j}\|_q \quad {\rm and \ } \lim_{q\ra 2} \del(q)=0.
\ee
It follows then from (\ref{W3}), (\ref{X3}) that
\be \label{Y3}
\| f_{r,j}\|_q \ \le \  Cr\|D\tilde{\mathbf{b}}(\cdot)\|_\infty\|h_{Q,j}\|_q[1+\del(q)]^{r-2}(1-\la/\La)^{r-2} |v| \ ,
\ee
where $C$ depends only on $d$. Now Young's inequality and (\ref{X3}), (\ref{Y3}) imply that
\begin{multline}  \label{Z3}
\sum_{z\in Q}\big\|\sum_{x\in Q} g(x)e^{-iz\cdot\xi} \tilde{\mathbf{b}}(\phi(x-z))\na u_{r,j}(x-z,\phi(\cdot))\big\|^2  \\
\le \ C\left[r\|D\tilde{\mathbf{b}}(\cdot)\|_\infty \|g\|_p\|h_{Q,j}\|_q [1+\del(q)]^{r-1}(1-\la/\La)^{r-1}|v|\right]^2  \ ,
\end{multline}
where $p=2q/(3q-2)$ with $1\le q\le 2$ and $C$ depends only on $d$.  We can argue now as in the $r=1$ case to establish (\ref{A3}) for $r\ge 2$ by choosing $q<2$ to satisfy $[1+\del(q)](1-\la/\La)\le (1-\la/\La)^{1/2}$. 
\end{proof}

\end{document}